\documentclass[a4paper,fleqn]{cas-sc}

\usepackage[numbers,sort&compress]{natbib}
\usepackage{graphicx}
\usepackage{amsthm}
\usepackage{amsmath}
\usepackage{amsfonts}
\usepackage{amssymb}
\usepackage{mathrsfs}
\usepackage{geometry}
\usepackage{color}
\usepackage{placeins}
\usepackage{float}
\newcommand\br{\mathbb{R}}

\newcommand\dd{\,\mathrm{d}}
\def\tsc#1{\csdef{#1}{\textsc{\lowercase{#1}}\xspace}}
\tsc{WGM}
\tsc{QE}

\newtheorem{theorem}{Theorem}[section]

\newdefinition{rmk}{Remark}
\newtheorem{prop}[theorem]{Proposition}

\begin{document}
\let\WriteBookmarks\relax
\let\printorcid\relax
\def\floatpagepagefraction{1}
\def\textpagefraction{.001}

\shorttitle{Symmetric Exponential Integrator for RCPD}    

\shortauthors{Z. R. Shen and B. Wang}  

\title [mode = title]{An Explicit Symmetric Exponential Integrator and Its Error Estimate for the Relativistic Charged-Particle Dynamics}  

\tnotemark[1] 

\tnotetext[1]{This work was supported partially by the National Natural Science Foundation of China (Grant No. 12371403).} 

%

\author[1]{Zhirui Shen}



\ead{zrshen@stu.xjtu.edu.cn}



\affiliation[1]{organization={School of Mathematics and Statistics},
            addressline={Xi'an Jiaotong University}, 
            city={Xi'an},
            postcode={710049}, 
            country={China}}

\author[1]{Bin Wang}
\cormark[1]

\ead{wangbinmaths@xjtu.edu.cn}




\cortext[1]{Corresponding author}



\begin{abstract}
This paper investigates the equations of motion for a relativistic charged particle in a general magnetic field. By reformulating the dynamics in four-dimensional spacetime and separating the linear and nonlinear parts, we construct an explicit symmetric exponential integrator based on Lie splitting. Rigorous analysis establishes its unconditional stability and second-order convergence. Numerical experiments confirm its superior performance, including accuracy, effciency and long-time Hamiltonian conservation.\nocite{*}
\end{abstract}




\maketitle

\section{Introduction}\label{intro}

In this paper, we consider the equations of motion of a relativistic charged particle of the form
\begin{equation}\label{3D}
    \boldsymbol x'(\bar{t})=\frac{\boldsymbol p (\bar{t})}{\gamma},\quad \boldsymbol p'(\bar{t})=\frac{\boldsymbol p(\bar{t})}{\gamma}\times\boldsymbol B(\boldsymbol x(\bar{t}))+\boldsymbol E(\boldsymbol x(\bar{t})),\quad\gamma=\sqrt{1+|\boldsymbol p|^2},
\end{equation}
where $\boldsymbol x\in\br^3$ is the position at time $\bar t$, $\boldsymbol u\in\br^3$ is the momentum, and $\gamma$ is the relativistic factor. Here, $\boldsymbol B(\boldsymbol x)=(B_1,B_2,B_3)(\boldsymbol x)$ and $\boldsymbol E(\boldsymbol x)=(E_1,E_2,E_3)(\boldsymbol x)$ are respectively the magnetic and electric field at $\boldsymbol x$. It is governed by the relativistic Lorentz force equations, which form a nonlinear system of ordinary differential equations endowed with intrinsic geometric structures, such as time-reversal symmetry, phase-space volume preservation, and Lorentz covariance \cite{jackson1999classical,rohrlich2008dynamics}.

In recent years, symmetric numerical methods characterized by self-adjoint discrete flows, are known to possess favorable long-time behavior and reduced numerical dissipation \cite{benettin1994hamiltonian,hairer2006geometric}. For charged-particle dynamics, symmetric multistep and one-step methods have been studied in depth, together with rigorous long-term error analysis and near-conservation results \cite{hairer2017symmetric,hairer2020long,hairer2022large}. In parallel, exponential integrators and splitting methods have emerged as powerful tools for systems with stiff or oscillatory components \cite{mclachlan2002splitting}. They have been successfully applied to charged-particle and kinetic equations, including uniformly accurate methods for Vlasov-type models \cite{crouseilles2017uniformly,chartier2019uniformly,chartier2020uniformly} and error analysis of splitting schemes under magnetic effects \cite{wang2021error}. However, most existing exponential integrators for charged-particle dynamics are not designed to be symmetric.

From a relativistic perspective, the equations of motion admit a natural formulation in four-dimensional (4D) spacetime, which provides a convenient framework for incorporating Lorentz covariance into numerical algorithms. This viewpoint has motivated the development of covariant symplectic and volume-preserving methods for relativistic charged-particle dynamics \cite{wang2016lorentz,wang2021high,Zhang2024strongmagnetic}.
In the broader context of geometric numerical integration, symplectic and variational integrators based on discrete Lagrangian or Hamiltonian formulations \cite{maeda1982lagrangian,qin2008variational,tao2016explicit,zhang2018explicit}, as well as volume-preserving algorithms \cite{feng1995volume,he2015volume,zhang2015volume,he2016high,matsuyama2017high}, have been extensively studied.
Nevertheless, symmetric exponential integrators (sEI) derived from the 4D equations of motion, without assuming a strong magnetic field regime, remain comparatively scarce.

In this work, we focus on the general magnetic field. By reformulating the relativistic equations of motion in a 4D form and separating the linear and nonlinear components, we construct an exponential integrator of (\ref{4D}) with some remarkable features as
\begin{enumerate}
    \item It is a symmetric method of second order.
    \item The method is based on a simple composition of Lie splitting and features a concise form.
    \item The method is explicit and computationally efficient.
\end{enumerate}

The remainder of the paper is organized as follows. In Section \ref{brief}, we briefly review the equations of motion of relativistic charged particle dynamics in 4D formulations and present the specific form of our explicit symmetric exponential integrator (sEI). Section \ref{converge} provides the stability and convergence analysis of sEI. The numerical experiments in Section \ref{numerical} illustrate the performance of the new integrators. Finally, some concluding remarks are drawn in Section \ref{conclusion}.

\section{Main results}\label{main results}
\subsection{A brief introduction to the 4D equations and sEI}\label{brief}
Let $\tau$ be the proper time as the parameter for the particle's worldline with $\frac{\dd\bar{t}}{\dd\tau}=\gamma$, through the coordinate transformation $\boldsymbol v=\boldsymbol p$, $w=i\gamma$ and $t=i\bar{t}$, we represent the time and space coordinates of (\ref{3D}) as a four-dimensional spacetime (cf.\cite{Zhang2024strongmagnetic}):
\begin{equation}\label{4D}
    \begin{aligned}\dot{\boldsymbol{x}}(\tau)&=\boldsymbol{v}(\tau),\\\dot{t}\left(\tau\right)&=w(\tau),\\\dot{\boldsymbol{v}}(\tau)&=-iw(\tau)\boldsymbol E(\boldsymbol x(\tau))+v(\tau)\times\boldsymbol{B}(\boldsymbol x(\tau)),\\\dot{w}(\tau)&=i\boldsymbol E(\boldsymbol x)\cdot v(\tau).\end{aligned}   
\end{equation}
Both the variables $t$ and $w$ are cure imaginary numbers, whereas  $\boldsymbol y=(\boldsymbol x,t)^T$ is called 4-position and $\boldsymbol u=(\boldsymbol v ,w)^T$ is called 4-velocity. By defining the following two matrices
\[
    \widehat{B}(\boldsymbol{x}):=\begin{pmatrix}0&B_3&-B_2&0\\-B_3&0&B_1&0\\B_2&-B_1&0&0\\0&0&0&0\end{pmatrix},\quad
     \widehat{E}(\boldsymbol{x})=\begin{pmatrix}0&0&0&-iE_1\\0&0&0&-iE_2\\0&0&0&-iE_3\\iE_1&iE_2&iE_3&0\end{pmatrix},
\]
we can further rewrite (\ref{4D}) as
\begin{equation}\label{origin}
\frac{\dd}{\dd \tau}\begin{pmatrix}
 \boldsymbol x\\
 t\\
 \boldsymbol v\\
 w
\end{pmatrix}=\begin{pmatrix}
 0 & I_4\\
 0 & \widehat{B}(\boldsymbol{x_0})
\end{pmatrix}\begin{pmatrix}
 \boldsymbol x\\
 t\\
 \boldsymbol v\\
 w
\end{pmatrix}+\begin{pmatrix}
 0\\
 0\\
G(\boldsymbol{x})\begin{pmatrix}
 \boldsymbol  v\\
w
\end{pmatrix}
\end{pmatrix},
\end{equation}
where
\(
    G(\boldsymbol{x})=\widehat{B}(\boldsymbol{x})-\widehat{B}(\boldsymbol{x_0})+\widehat{E}(\boldsymbol{x}).
\)
Setting $U=(\boldsymbol{x},t,\boldsymbol{v},w)^T=(\boldsymbol{y},\boldsymbol{u})^T$, we write (\ref{origin}) as
\begin{equation}\label{U'=LU+F(U)}
    U'=LU+F(U),\quad\mathrm{with}\quad L=\begin{pmatrix}
 0 & I_4\\
 0 & \boldsymbol B(\boldsymbol{x_0})
\end{pmatrix}\ \mathrm{and}\ F(U)=(0,0,G(\boldsymbol{x})\boldsymbol u)^T.
\end{equation}

Suppose $h$ be the time-step and $U^n$ be the $n$-th numerical solution of (\ref{U'=LU+F(U)}). Letting $\mathcal U^{n+1}=(U^{n+1},U^{n})^T$, the sEI can be written as
\begin{equation}\label{sEI}
    \begin{aligned}\mathcal U^{n+1}=&\Phi_h(\mathcal U^{n})=\begin{pmatrix}0&e^{2hL}\\1&0\end{pmatrix}\mathcal U^{n}+\begin{pmatrix}2he^{hL}F(U^n)\\0\end{pmatrix},\quad n\geq1,\\U^{1}=&e^{hL}U^{0}+\tau e^{hL}F(U^{0}).\end{aligned}
\end{equation}
The matrix exponential $e^{hL}$ can be caculated by $e^{hL}=\sum_{n=0}^{\infty}(hL)^n/n!$ and we get
\begin{equation}
e^{hL} = \begin{pmatrix}
I_4 &  \begin{pmatrix} h\varphi_1(h S) & 0 \\ 0 & h \end{pmatrix} \\[6pt]
0 & \begin{pmatrix} e^{h S} & 0 \\ 0 & 1 \end{pmatrix}
\end{pmatrix}\quad\mathrm{with}\quad \varphi_1(hS)=\int_0^1e^{\sigma hS}\dd\sigma,\quad S = \begin{pmatrix}
0 &  {B}_3 & - {B}_2 \\
- {B}_3 & 0 &  {B}_1 \\
 {B}_2 & - {B}_1 & 0
\end{pmatrix}.
\end{equation}
The block matrices $e^{tS}$ and $\varphi_1(tS)$ can be computed using Rodrigues formula, yielding
\[
e^{tS} = I_3 + \frac{\sin(t\theta)}{\theta} S + \frac{1 - \cos(t\theta)}{\theta^2} S^2,\quad \varphi_1(tS) = I_3 + \frac{1 - \cos(t\theta)}{t \theta^2} S + \frac{1}{\theta^2} \left(1 - \frac{\sin(t\theta)}{t \theta} \right) S^2,
\]
where $\theta = \sqrt{ {B}_1^2 +  {B}_2^2 +  {B}_3^2}$. 
\begin{rmk}\label{matrix norm}
The above computations yield some norm estimates for $e^{hS}$, $\varphi_1(hS)$ and $e^{hL}$. Since $S$ is a skew-symmetric matrix, $e^{hS}$ is orthogonal, and hence $$||e^{hS}||\le1.$$ Next, note that $||\varphi_1(hS)\boldsymbol v||\le \int_0^1||e^{\sigma hS}\boldsymbol v||\dd\sigma\le||\boldsymbol v||$ holds for any $\boldsymbol v\in\br^3$, thus $$||\varphi_1(hS)||\le1.$$ Moreover, for any $ V=(\boldsymbol a,\boldsymbol b)^T\in\br^4\times\br^4$, $||e^{hL}V||^2\le ||\boldsymbol a+h\boldsymbol b||^2+||\boldsymbol b||^2\le(1+h)^2(||\boldsymbol a||^2+||\boldsymbol b||^2)$, that is to say $$||e^{hL}||\le1+h.$$
\end{rmk}

\subsection{Convergence analysis for sEI}\label{converge}
We first proceed to prove the unconditional stability of the sEI.
\begin{prop}[Stability]\label{stable}
Suppose $V_1,V_2,W_1,W_2 \in C^1(\br^4\times\br^4)$ and $h\in(0,1]$, and denote $\mathcal V=(V_1,V_2)^T$ and $\mathcal W=(W_1,W_2)^T$. Then the sEI (\ref{sEI}) satisfies, with $C>0$ independent to $h$,
\begin{equation}
    ||\Phi_h(\mathcal V)-\Phi_h(\mathcal W)||\le (1+Ch)||\mathcal V-\mathcal W||.
\end{equation}
\end{prop}
\begin{proof}
    By the Remark \ref{matrix norm} we have
    \begin{align*}
        ||\Phi_h(\mathcal V)-\Phi_h(\mathcal W)||&\le ||(e^{2hL}(V_2-W_2),V_1-W_1)^T||+2h||e^{hL}(F(V_2)-F(W_2))||\\
        &\le (1+2h)||\mathcal V-\mathcal W||+2(2||\widehat{B}||+||\widehat{E}||)h||\mathcal V-\mathcal W||.
    \end{align*}
    Hence, taking $C=2(1+2||\widehat{B}||+||\widehat{E}||)$ gives the required constant.
\end{proof}
With the stability of sEI established above, we can now prove its second-order convergence.
\begin{theorem}
    Let \(U^n\) be the numerical solution of the sEI (\ref{sEI}) with 
    \(\boldsymbol B, \boldsymbol E \in C^2(\mathbb{R}^3;\br^3)\), for solving (\ref{origin}) up to some fixed \(T > 0\), and let \(U(\tau_n)\) be the exact solution of system (\ref{U'=LU+F(U)}) at \(\tau_n\). For the method with the time step \(0 < h \leq 1\), it has the following error bound
    \begin{equation}
        ||U(\tau_n)-U^n||\le C h^2,\quad0\leq n\leq T/h,
    \end{equation} 
    where the constant $C$ is independent of $n, h$ but depends on $T$.
\end{theorem}

\begin{proof}
    The variation-of-constant formula to (\ref{U'=LU+F(U)}) is (we omit $\boldsymbol x$ for brevity from now on) 
\begin{equation}\label{vcf}
    U(\tau_{n}+s)=e^{s L}U(\tau_n)+\int_{0}^{s }e^{(s-\sigma)L}F(U(\tau_n+\sigma))\dd\sigma,\ \forall s\in(0,h].
\end{equation}
Multiplying $e^{s L}$ on both sides of (\ref{vcf}) and taking $s=h$, $s=-h$, we have
\begin{align*}
        e^{-hL}U(\tau_{n+1})&=U(\tau_n)+\int_{0}^{\tau }e^{-sL}F(U(\tau_n+s))\dd s,
     \\
        e^{hL}U(\tau_{n-1})&=U(\tau_n)-\int_{-\tau}^{0 }e^{-sL}F(U(\tau_n+s))\dd s. 
\end{align*}
Subtracting the above two equations gives
\[
e^{-hL}U(\tau_{n+1})-e^{hL}U(\tau_{n-1})=\int_{-\tau }^{\tau }e^{-sL}F(U(\tau_n+s))\dd s.
\]
Multiplying $e^{hL}$ on both sides, we have
\begin{equation}\label{main}
    U(\tau_{n+1})=e^{2h L}U(\tau_{n-1})+2he^{hL}F(U(\tau_n))+e^{hL}\int_{-h }^{h }e^{-sL}F(U(\tau_n+s))-F(U(\tau_n))\dd s.
\end{equation}
Define $\xi_n(s):=e^{sL}F(U(\tau_n+s))$. By Taylor's theorem, the local truncation error is 
\begin{equation*}
    \mathcal E_{n+1}^h:=e^{hL}\left[\int_{-h}^h\xi_n(s)-\xi_n(0)\dd s\right]=e^{hL}\int_{-h}^{h}\frac{s^2}{2}\xi''(\sigma)\dd s=\frac{h^3}{3}e^{hL}\xi''(\sigma),\quad\sigma\in[-h,h].
\end{equation*}
Since $\boldsymbol B$ and $\boldsymbol E$ are both $C^2$ functions, $F$ is also $C^2$, consequently $\bigl\{\xi_n''(s)\bigr\}_{n=1}^{[T/h]}$  is uniformly bounded on $[-h,h]$. Moreover, since $||e^{hL}||\le 2$, we have $||\mathcal E_n^h||\le C h^3$. 

Letting $\mathcal U(\tau_{n})=(U(\tau_n),U(\tau_{n-1}))^T$ and according to the Proposition \ref{stable}, there exists a positive number $C_1$ independent to $h$ such that
\begin{equation}\label{gronwall}
    ||\mathcal U(\tau_{n+1})-\mathcal U^{n+1}||\le ||\Phi_h(\mathcal U(\tau_n))-\Phi_h(\mathcal U^n)||+2||\mathcal E_{n+1}^h||\le (1+Ch)||\mathcal U(\tau_{n})-\mathcal U^{n}||+C_1h^3.
\end{equation}
Applying the discrete Gronwall's inequality to $\bigl\{||\mathcal U(\tau_{n})-\mathcal U^{n}||\bigr\}_{n=1}^{\infty}$ in (\ref{gronwall}) yields
\begin{equation*}
    ||U(\tau_n)-U^n||\le ||\mathcal U(\tau_{n})-\mathcal U^{n}||\le C h^2.
\end{equation*}
\end{proof}

\section{Numerical Results}\label{numerical}
In this section, we report a numerical experiments of the relativistic dynamics of charged particles by using our presented sEI in Section \ref{brief}. The numerical results demonstrate its superb properties in accuracy and efficiency. To clearly analyze the global error at $T=1$, we introduce the following error function to illustrate the convergence orders of $U$.
\[
\mathrm{err}_{U}(\tau_n):=\frac{\|\boldsymbol y^n-\boldsymbol y(\tau_n)\|}{\|\boldsymbol y(\tau_n)\|}+\frac{\|\boldsymbol u^n-\boldsymbol u(\tau_n)\|}{\|\boldsymbol u(\tau_n)\|}.
\]
More precisely, in equation (\ref{origin}), we set the electromagnetic field as
\[
\boldsymbol B\left(\boldsymbol x\right)=\begin{pmatrix}\cos\left(x_2\right)\\1+\sin\left(x_3\right)\\\cos\left(x_1\right)\end{pmatrix}+\begin{pmatrix}-x_1\\0\\x_3\end{pmatrix},\quad\mathrm{and}\quad \boldsymbol E(\boldsymbol x)=-\nabla V(\boldsymbol x),\quad V(\boldsymbol x)=\left(x_1^2+x_2^2\right)^{-\frac12}.
\]
We consider two sets of initial: (I). $\boldsymbol y_0=\left(\frac13,\frac14,\frac12,0\right)^T$, $\boldsymbol u_0=\left(\frac25,\frac23,1,0\right)^T$ and (II). $\boldsymbol y_0=\left(0,1,0.1,0\right)^T$, $\boldsymbol u_0=\left(0.09,0.05.0.2,0\right)^T$. And we take the time step as $h=2^{-j}$ for $j=5,6,\cdots,10$. The error results are presented for initial data (I) and (II) respectively in Figure \ref{2order}. It can be observed that sEI indeed exhibits second-order accuracy.
\begin{figure}[h!]
    \centering
    \begin{minipage}{0.35\textwidth}
        \centering
        \includegraphics[scale=0.45]{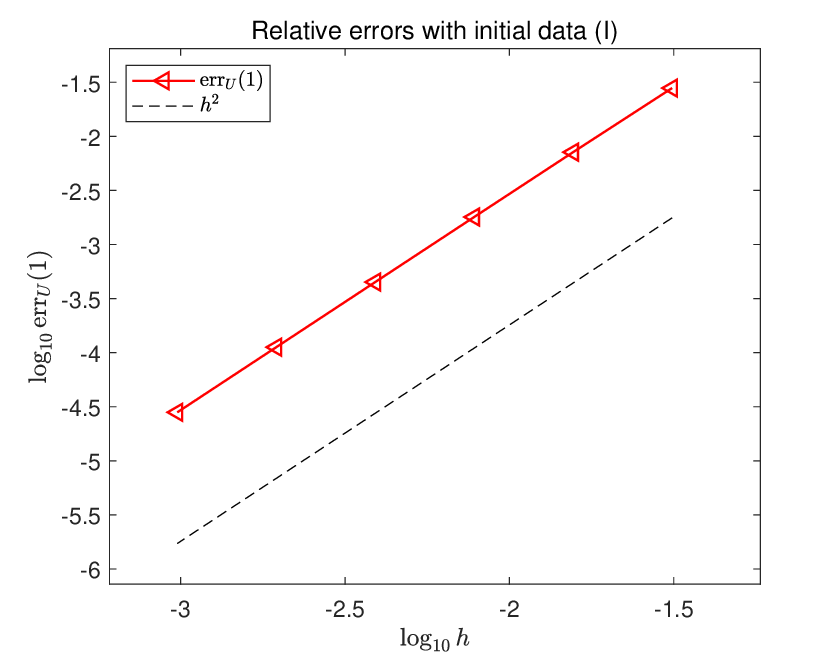}
    \end{minipage}
    \thinspace
    \begin{minipage}{0.35\textwidth}
        \centering
        \includegraphics[scale=0.45]{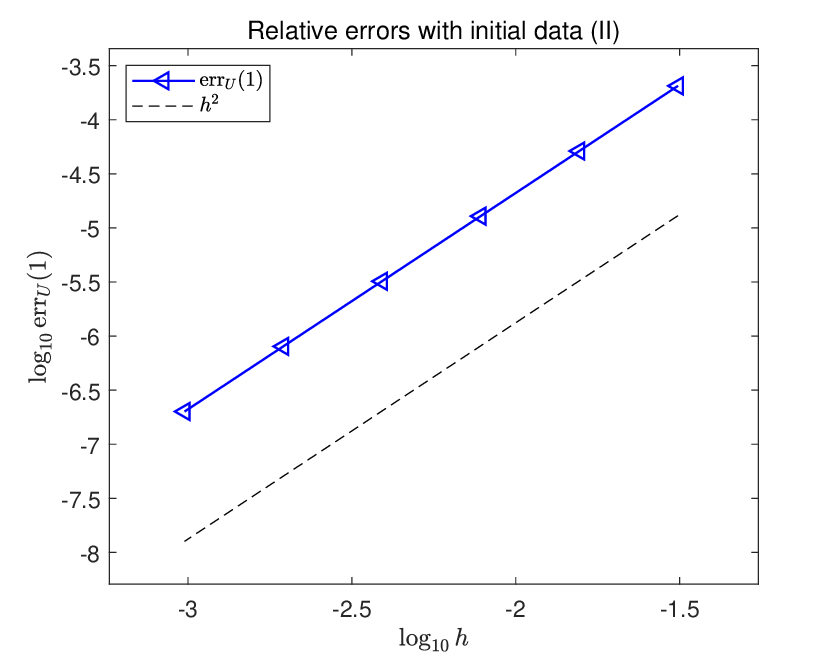}
    \end{minipage}    
    \caption{Global errors $\mathrm{err}_U(1)$ under initial conditions (I) and (II).}\label{2order}
\end{figure}

\begin{figure}[h!]
    \centering
    \begin{minipage}{0.35\textwidth}
        \centering
        \includegraphics[scale=0.45]{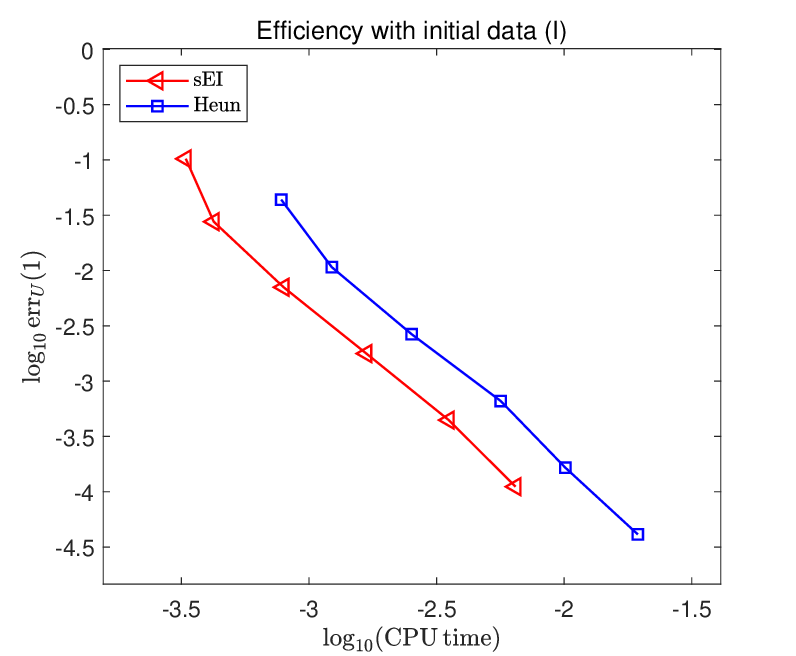}
    \end{minipage}
    \thinspace
    \begin{minipage}{0.35\textwidth}
        \centering
        \includegraphics[scale=0.45]{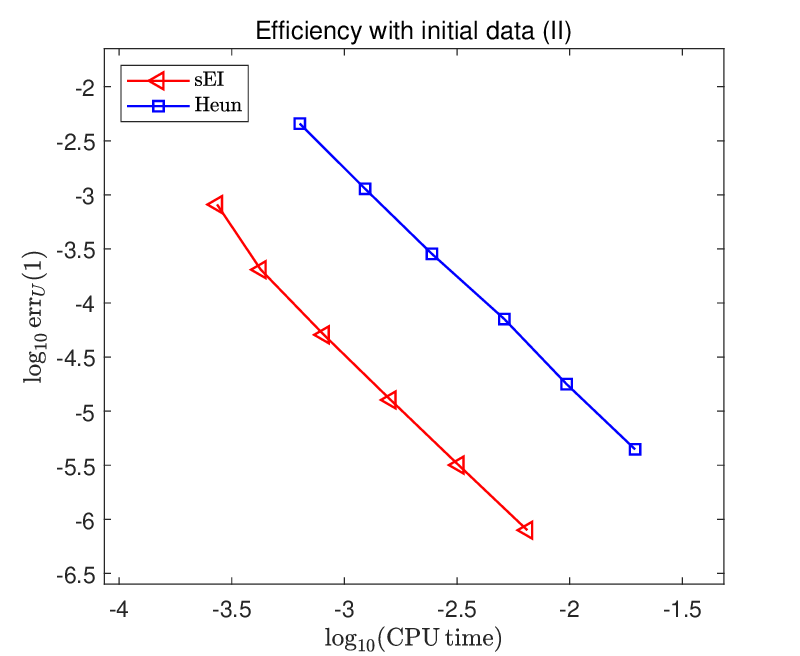}
    \end{minipage}    
    \caption{The comparison of the CPU time between sEI and Heun under initial conditions (I) and (II).}\label{CPU time}
\end{figure}

\begin{figure}[h!]
    \centering
    \begin{minipage}{0.35\textwidth}
        \centering
        \includegraphics[scale=0.45]{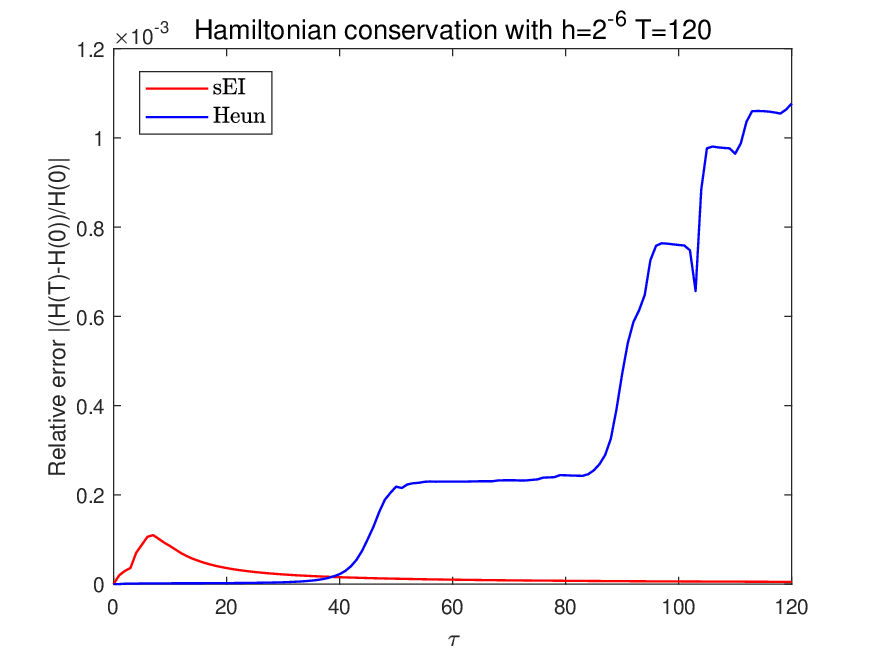} 
    \end{minipage}
    \hspace{0.02\textwidth}
    \begin{minipage}{0.35\textwidth}
        \centering
        \includegraphics[scale=0.45]{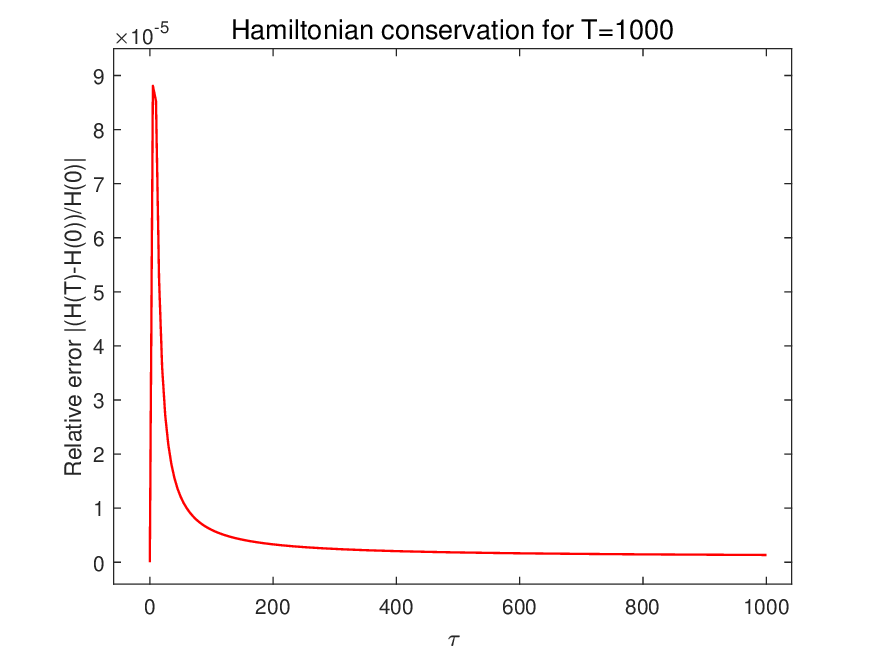}
    \end{minipage}    
    \caption{The relative errors of Hamiltonian for sEI and  Heun method with $h=2^{-6}$.}\label{Hamilton}
\end{figure}

To investigate the efficiency of the sEI, we compare it with the Heun method, a explicit second-order Runge-Kutta scheme. In Figure \ref{CPU time}, owing to the concise formulation of sEI, the proposed method possesses a clear advantage in terms of computational efficiency.

Finally, we discuss the conservation of the Hamiltonian $H(\boldsymbol x,w)=V(\boldsymbol x)-iw$ of (\ref{origin}) by the sEI. In Figure \ref{Hamilton}, we adopt the initial value (II) and the terminal time $T=120$ to investigate the relative errors $\bigg|\frac{H(T)-H(0)}{H(0)}\bigg|$ of the Hamiltonian for both the sEI and the Heun method, as well as the long-time conservation behavior of the Hamiltonian by the sEI method up to the terminal time $T=1000$. It is evident that our sEI exhibits excellent near-conservation over long times, whereas the Heun method does not.
\FloatBarrier 
\section{Conclusion}\label{conclusion}
We have proposed an explicit symmetric exponential integrator for relativistic charged-particle dynamics in general magnetic fields. The scheme is derived from a 4D spacetime formulation of the equations of motion and is built on a straightforward Lie splitting method. The method is explicit, symmetric, and second-order accurate, demonstrating superior efficiency and remarkable near-conservation of the Hamiltonian. Theoretical analysis confirms its unconditional stability and second-order convergence. Future work will extend the approach to time-dependent or self-consistent field models, as well as to higher-order symmetric or geometric variants.

\bibliographystyle{cas-model2-names}

\bibliography{cas-refs}



\end{document}